\documentclass[a4paper]{amsart}
\usepackage[utf8]{inputenc}
\usepackage{array,epsfig}
\usepackage{amsmath}
\usepackage{amsfonts}
\usepackage{amssymb}
\usepackage{amsxtra}
\usepackage{amsthm}
\usepackage{mathrsfs}
\usepackage{float}
\usepackage{color}
\usepackage{relsize}
\usepackage{mathrsfs}
\usepackage{environ}
\usepackage{caption}
\usepackage[hidelinks]{hyperref} 
\usepackage{CJKutf8}
\usepackage{mathtools}
\newcommand{\Ric}{\textnormal{Ric}}
\newcommand{\vol}{\textnormal{vol}}

\newcounter{def}
\newtheorem{thm}{Theorem}[section]

\theoremstyle{definition}
\newtheorem{conj}{Conjecture}

\theoremstyle{remark}

\theoremstyle{definition}
\newtheorem{definition}[def]{Definition}

\newtheorem*{problem*}{Problem}
\newcommand{\norm}[1]{\left\lVert#1\right\rVert}

\begin{document}	
	\title{Proof of Bishop's Volume Comparison Theorem Using Singular Soap Bubbles}
	\author{H. BRAY, F. GUI, Z. LIU, AND Y. ZHANG}
	\begin{abstract}

		Bishop's volume comparison theorem states that a compact $n$-manifold with Ricci curvature larger than the standard $n$-sphere has less volume. While the traditional proof uses geodesic balls, we present another proof using isoperimetric hypersurfaces, also known as ``soap bubbles,'' which minimize area for a given volume. Curiously, isoperimetric hypersurfaces can have codimension 7 singularities, an interesting challenge we are forced to overcome.
	\end{abstract}
	\maketitle
	
	\section{Introduction}
	The following Bishop's theorem is a classic volume comparison theorem in Riemannian geometry.
	\begin{thm}[Bishop's theorem] Let $(S^n,g_0)$ be an n-sphere with standard metric $g_0$ and Ricci curvature $\Ric_0\cdot g_0$. Let $(M,g)$ be a compact connected smooth Riemannian manifold of dimension $n\ge 3$ without boundary. Let $\Ric(g)$ and $\textnormal{vol}(M)$ denote the Ricci curvature and the volume of $M$, respectively.
		If $\Ric(g)\ge \Ric_0\cdot g_0$, then we have ${\vol}(M)\le ${\vol}$(S^n)$.
		\begin{figure}[H]
			\centering
			\includegraphics[width=0.7\linewidth]{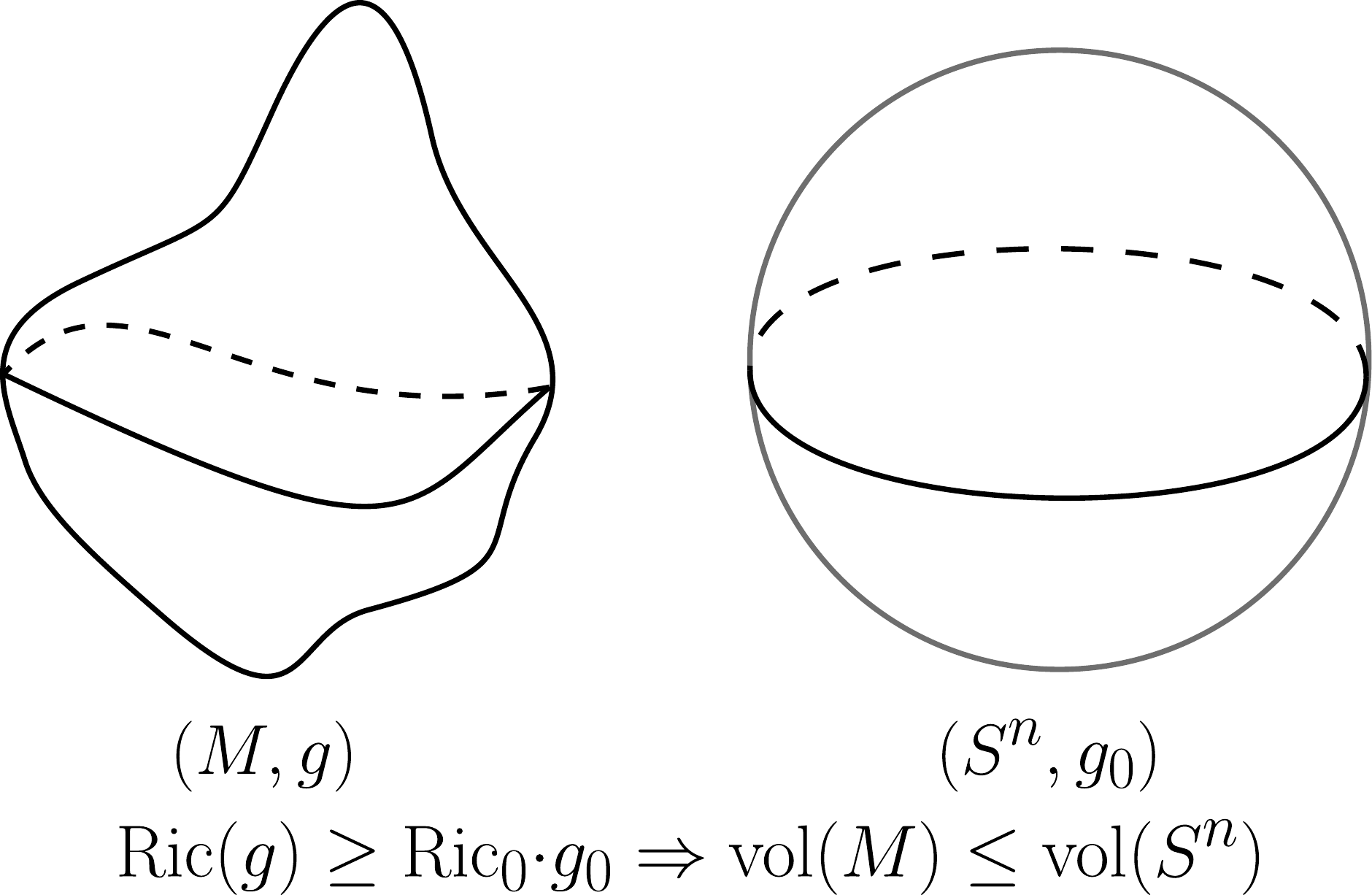}
			\caption{Bishop's Thoerem}
			\label{fig:figure-1-bishop-thoerem}
		\end{figure}

	\end{thm}
	This theorem was first proven by Bishop in 1963 \cite{bishop}.
	A standard proof using geodesic balls  can be found in \cite{petersen2016}.
	
	In this paper, we follow the ideas developed  in the first author's thesis \cite{bray2009penrose} and give a new proof of this well-known result using isoperimetric hypersurfaces as defined below and geometric measure theory.
	\begin{definition} 
		Let $\Sigma=\partial W\subset M$ be a compact hypersurface in $M$.  $\Sigma$ is called an isoperimetric surface of $M$, if $\Sigma$ is the area minimizer among all the hypersurfaces bounding the same volume $V$. We can also call $\Sigma$ a soap bubble.
	\end{definition}
	According to the great review by Antonio Ros in \cite{anros}, the existence of an isoperimetric surface $\Sigma$ for any given $V\in (0,\textnormal{vol}(M))$ in a compact manifold $M$ can be found in the monograph \cite{almgren1976existence} by Almgren.
	
	How does our proof work?	Roughly speaking, we turn the Ricci curvature bound into a second order ordinary differential inequality about the area function of isoperimetric hypersurfaces with volume parameter. Since the entire manifold itself can be realized as the inside of an isoperimetric hypersurface with zero area bounding the largest volume, the inequality we obtained from the Ricci curvature bound can be utilized to get an upper bound for the volume. In Section 2, we present a detailed exposition of this idea. 
	
	However, one technicality of isoperimetric hypersurfaces is that they may have singularities in dimensions larger than seven. Thus, the original idea \cite{bray2009penrose} can work without modification only for dimensions less than eight. Yet heuristically if the singular part is small enough, then the method should still work. This is how geometric measure theory comes into play. In Section 3 and 4, we use geometric measure theory to control the size of the singularities  and finish the proof of the theorem in all dimensions. 
	
	In Section 5, we discuss a scalar curvature volume comparison theorem in dimension $3$. The theorem can be regarded as a refined version of Bishop's theorem, since we also incorporate the information about scalar curvature into the volume bound. The only known proof of this scalar curvature volume comparison theorem uses the isoperimetric techniques described in this paper.
	
	\section{Smooth Case of Bishop's Theorem}
	
	This section will give a brief review of isoperimetric surface techniques and Bray's proof of Bishop's theorem in dimensions less than 8. These techniques will fail in higher dimensions since isoperimetric hypersurfaces may exhibit singularities. This issue is addressed in section 3 and 4. In this section, we will only be dealing with manifolds with dimensions less than 8, unless otherwise stated.
	
	\subsection{Isoperimetric Profile Function}
	\begin{definition}
		Let the isoperimetric profile function of $(M, g)$ be
		\[A(V) = \inf_{R}\,\{\textnormal{area}(\partial R)\,\mid\,\vol(R) = V\}\]
		where $R$ is any region in $M$, $\textnormal{area}$ is the $n-1$ dimensional Hausdorff measure and $\vol$ is the $n$ dimensional Hausdorff measure. If there exists a region $R$ that minimizes this quantity, then we say $\Sigma = \partial R$ minimizes area with the given volume constraint. 
	\end{definition}
	
	Note that, in general, minimizers may not exist and when they exist, it may not be unique. However, in the case of Bishop's theorem, we are dealing with manifolds that have $\Ric(g) \geq \Ric_0 > 0$. As mentioned in the introduction, there always exists a minimizer for any given $V\in (0, \textnormal{vol}(M))$. We will only be dealing with manifolds in dimension less than $8$ in this section. So according to lemma 3.1 which we will state later,  the isoperimetric hypersurfaces will always be smooth and have constant mean curvature.
	
	The goal is to use function $A(V)$ to get an upper bound for $\vol(M)$. Notice that $A(V)$ has two roots, 0 and $\vol(M)$. Let $A_M(V)$ be the $A(V)$ function on $(M, g)$ and $A_S(V)$ be the corresponding function on $(S^n, g_0)$ where $S^n$ is the standard sphere with scaling and $\Ric(g) \geq \Ric_0\cdot g$. We want to show that $A_M(V)$ reaches its second root faster than $A_S(V)$. Figure \ref{IPF_2} gives an intuition of how we will prove Bishop's theorem. To that end, we need to understand how ``fast'' the function $A(V)$ curves, that is, the second derivative $A''(V)$. 
	\begin{figure}[tbh]
		\centering
		\includegraphics[width = 0.6\textwidth]{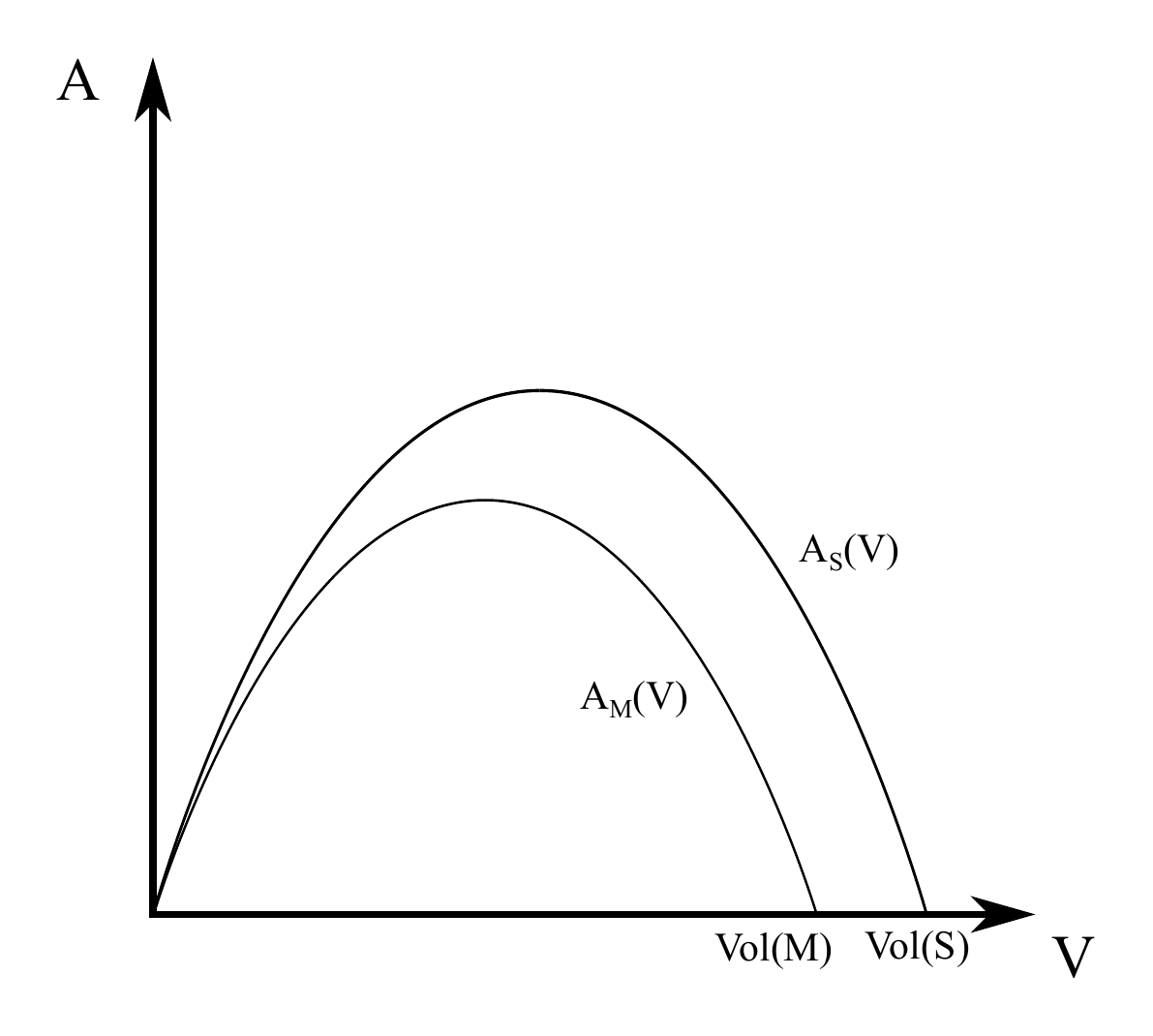}
		\caption{Isoperimetric profile function $A_M(V)$ lies below $A_S(V)$ and hence the root $\mathrm{vol}(M)\leq \mathrm{vol}(S^n)$}
		\label{IPF_2}
	\end{figure}
	Note that, $A''(V)$ may be not well defined. However, we try to establish an inequality for  $A''(V)$:
	\[A''(V)\leq -\frac{1}{A(V)}\left(\frac{1}{n-1}A'(V)^2+\Ric_0\right)\]
	in the sense of comparison function, which means for all $V_0\ge 0$, there exists a smooth function $A_0(V)$, such that $A_0(V)\ge A(V)$, $A_0(V_0)=A(V_0)$, 
	\[A''_0(V_0)\leq -\frac{1}{A_0(V_0)}\left(\frac{1}{n-1}A'_0(V_0)^2+\Ric_0\right).\]
	
	When the isoperimetric hypersurface is smooth, we can do a unit normal variation on $\Sigma(V)$. Fix $V = V_0$ and flow $\Sigma(V_0)$ along the outward-pointing unit normal vector $\nu$ for time $t$. Since $\Sigma(V)$ is smooth, the flow exists for $t\in (-\delta, \delta)$ for some $\delta>0$. 
	
	Let $\Sigma_{V_0}(t)$ be the surface at time $t$, which is the boundary of a region $R(t)$. Let $V = V(t)$ be its volume. With a slight abuse of notation, we parameterize by volume such that $\Sigma_{V_0}(V_0)$  corresponds to $\Sigma_{V_0}(t)$ at time $t = 0$. Denote $A_0(V)=\textnormal{area}(\Sigma_{V_0}(V))$ and denote  $A_{V_0}(t) = \textnormal{area}(\Sigma_{V_0}(t))$, so $A_0(V)$ and $A_{V_0}(t)$ are the same function with different parameters.  Then we have $A(V) \leq A_{0}(V)$ since $\Sigma_{V_0}(V)$ is not a minimizer of function $A(V)$. Hence, 
	\[A''(V_0) \leq A''_0(V_0).\]

	Figure \ref{IPF_1} shows the shapes of $A(V)$ and $A_{0}(V)$ in a neighborhood of $V_0$.
	\begin{figure}[tbh]
		\centering
		\includegraphics[width = 0.7\textwidth]{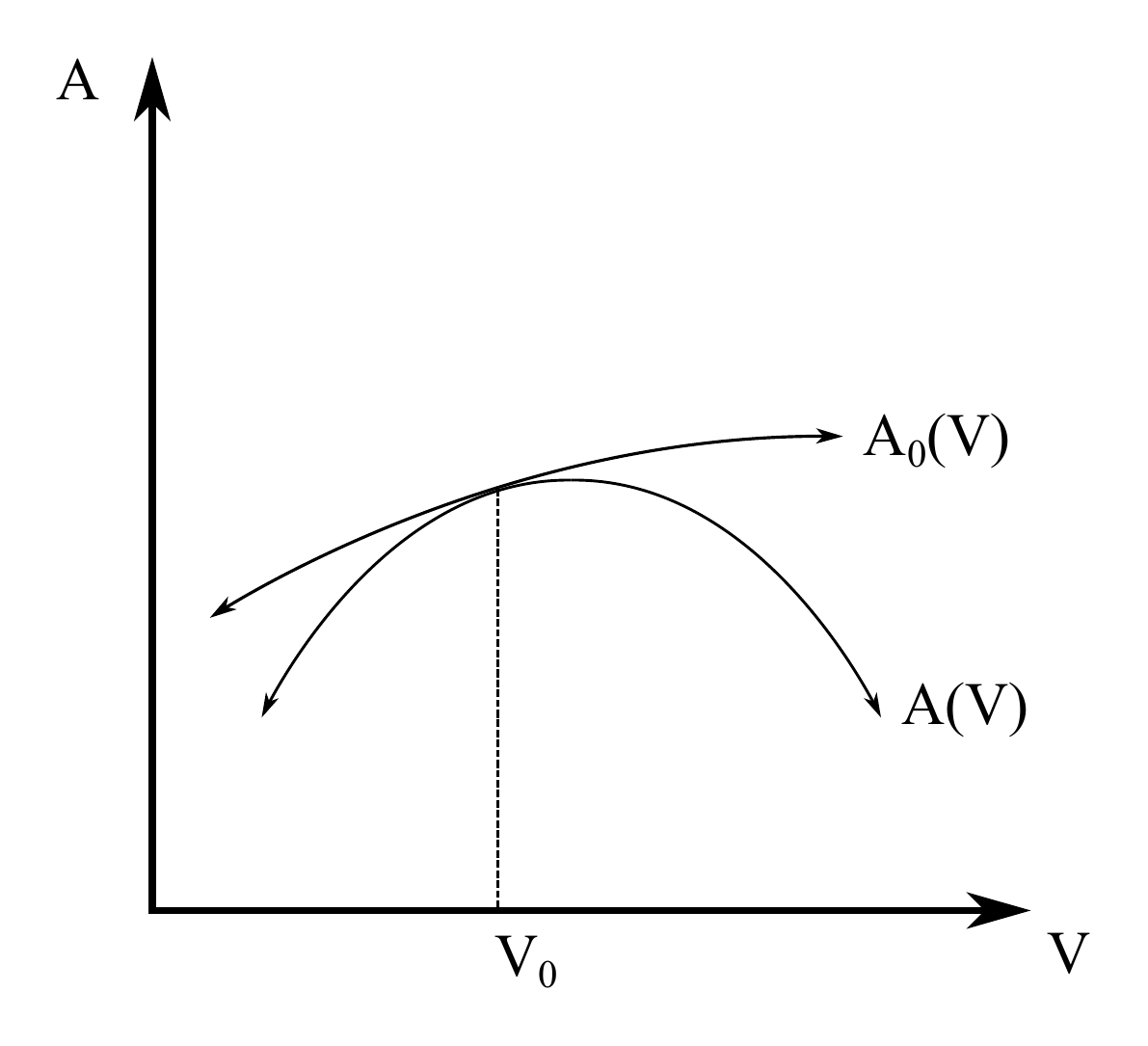}
		\caption{Isoperimetric profile function versus unit normal flow of the isoperimetric hypersurface at a fixed volume $V_0$}
		\label{IPF_1}
	\end{figure}
	
	To get a bound on $A_{0}(V_0)$, we use the unit normal variation. We know that 
	\[A_{V_0}(t) = \int_{\Sigma_{V_0}(t)}\;dA = V'(t)\]
	where $dA$ is the area $(n-1)$-form. By the first variation of volume formula and variation of mean curvature, we obtain that
	\[\dot{dA} = H\;dA\qquad \text{and} \qquad \dot{H} = -\norm{\Pi}^2 - \Ric(\nu, \nu)\]
	where $\Pi$ is the second fundamental form and $H = \mathrm{tr}(\Pi)$ is the mean curvature. Since the mean curvature is constant on a smooth isoperimetric hypersurface, we have
	\[A_{V_0}'(t) = \int_{\Sigma_{V_0}(t)}H\; dA,\]
	\[A_{0}'(V_0) = \left.\frac{A_{V_0}'(t)}{V'(t)}\right|_{t = 0} = H.\]
	
	By simple calculus,
	\[A''_{0}(V_0) = \left.\frac{A''_{V_0}(t)- V''(t)A'_{0}(V_0)}{V'(t)^2}\right|_{t = 0},\]
	\begin{equation}\label{2nd-var}
	A_{V_0}''(t) = \int_{\Sigma_{V_0}(t)}H^2 - \norm{\Pi}^2 - \Ric(\nu,\nu)\;dA,
	\end{equation}
	\[V''(t) = A_{V_0}'(t) = \int_{\Sigma_{V_0}(t)}H\;dA.\]
	
	Putting these together, we have
	\begin{align*}
	A''_0(V_0) &= \frac{1}{A_{0}(V_0)^2}\left(\int_{\Sigma_{V_0}(V_0)} H^2-\norm{\Pi}^2-\Ric(\nu,\nu)\; dA - H\cdot \int_{\Sigma_{V_0}(V_0)}H\;dA \right) \\
	&= \frac{1}{A_{0}(V_0)^2}\int_{\Sigma_{V_0}(V_0)} -\norm{\Pi}^2 - \Ric(\nu, \nu)\; dA.
	\end{align*}
	
	Notice that $\norm{\Pi}^2\geq \frac{1}{n-1}H^2$ and $\Ric(\nu,\nu) \ge \Ric_0$ in the case of Bishop's theorem. Since they are constants on $\Sigma(V_0)$, we can deduce that
	\[A''(V_0) \leq A_{0}''(V_0) \leq -\frac{1}{A_{0}(V_0)}\left(\frac{1}{n-1}A_{0}'(V_0)^2+\Ric_0\right).\]
	
	Now we can vary $V_0$ and obtain that
	
	\begin{equation}\label{ipf bound}
	A''(V)\leq -\frac{1}{A(V)}\left(\frac{1}{n-1}A'(V)^2+\Ric_0\right).
	\end{equation}
	
	\vspace{3pt}
	\subsection{Proof of Bishop's Theorem in Smooth Case} \hfill\vspace{3pt}
	
	The isoperimetric profile function is the key to our proof of Bishop's theorem. Inequality \ref{ipf bound} gives us an upper bound on its second derivative. We will show in later sections that this inequality still holds in higher dimensions. But first, we will prove Bishop's comparison theorem assuming inequality \ref{ipf bound} holds for all dimensions.
	
	\begin{proof}[Proof of Theorem 1]
		
		First, define $F(V) = A(V)^\frac{n}{n-1}$, so that $F(V)$ has the same unit as $V$. Equation \ref{ipf bound} then can be rewritten as,
		
		\begin{equation} \label{F'' bound}
		F''(V) \leq -\frac{n\cdot \Ric_0}{n-1}F(V)^{-\frac{n-2}{n}}. 
		\end{equation}
		
		Observe that the boundary of a region $R$ in $M$ is same as the boundary of its complement $M\setminus R$. Therefore $A(V) = A(\vol(M) - V)$ and the same equation holds for $F(V)$ as well. Then by symmetry and negativity of $F''$, we know that when $V\in [0,\frac{1}{2}\vol(M)]$, $F(V)$ is strictly increasing and $F'(\frac{1}{2}\vol(M)) = 0$.
		
		Now we define the Ricci curvature mass:
		\[m(V) = \left(n^2 {\omega_{n-1}}^{\frac{n}{n-2}} - F'(V)^2\right) - \frac{n^2\cdot \Ric_0}{n-1}F(V)^\frac{2}{n}\]
		where $\omega_{n-1}$ is the volume of the sphere $S^{n-1}$. Take the derivative,  
		\begin{align*}
		m'(V) &= -2F'(V)F''(V) -\frac{2n\cdot \Ric_0}{n-1}F(V)^{-\frac{n-2}{n}}F'(V) \\
		&= -2F'(V)\left(F''(V) + \frac{n\cdot \Ric_0}{n-1}F(V)^{-\frac{n-2}{n}}\right) \geq 0.
		\end{align*}
		
		Since $M$ is a smooth manifold, $F(V)\approx {\omega_{n-1}}^\frac{n}{n-1} \cdot V$ for small $V$ and thus, $F'(0) = n\cdot{\omega_{n-1}}^\frac{1}{n-1}$. Therefore we have $m(0) = 0$. By the nonnegativity of the first derivative, we can further deduce that $m(V)$ is nonnegative.
		
		Let's consider the phase space in the $x$-$y$ plane with $x = F(V)$ and $y = F'(V)$. Let $\gamma$ be the path in the phase space when $V$ goes from 0 to $\frac{1}{2}\vol(M)$. From previous discussions, we know that $F(0) = 0$, $F'(0) = n\cdot {\omega_{n-1}}^{\frac{1}{n-1}} = y_0$ and $F'(\frac{1}{2}\vol(M)) = 0$. So if we set $F(\frac{1}{2}\vol(M)) = x_0$, then $\gamma$ is a path from $(0,y_0)$ to $(x_0,0)$. Further notice that $F(V)$ is strictly increasing and $F'(V)$ is strictly decreasing when $V\in \left[0,\frac{1}{2}\vol(M)\right]$. Therefore we have
		\begin{equation}\label{vol-int}
		\frac{1}{2}\vol(M) = \int_{\gamma}\;dV = \int_{\gamma}\;\frac{dx}{y}. 
		\end{equation}
		
		Consider all paths that terminate at $(x_0,0)$. The path with the smallest $y$ value maximizes the right-hand side of equation \ref{vol-int}. However, if we rewrite inequality \ref{F'' bound} as
		\[y\cdot \frac{dy}{dx} \leq -\frac{n\cdot \Ric_0}{n-1}x^{-\frac{n-2}{n}},\]
		then the path with the smallest $y$ value has equality in the above inequality. This implies that $m'(V) = 0$ and $m(V) = m_0$ is a constant. 
		
		Furthermore, if we rewrite the definition of the mass function, we have
		\[y = \left(n^2{c_{n-1}}^\frac{2}{n} -m_0 - \frac{n^2\cdot \Ric_0}{n-1}x^{\frac{2}{n}}\right)^{\frac{1}{2}}.\]
		
		The value $m_0$ determines the value of $x_0$ is the termination on $x$-axis. With a change of variables, we can compute that 
		\[\sup_{\gamma}\int_\gamma \;\frac{dx}{y} = \sup_{m_0}\left(n^2{c_{n-1}}^\frac{2}{n} - m_0\right)^\frac{n-1}{2}\left(\frac{n-1}{n^2\cdot \Ric_0}\right)^\frac{n}{2}\cdot\int_0^1\left(1-z^\frac{2}{n}\right)^{-\frac{1}{2}}\; dz.\]
		
		So the smaller value of $m_0$ yields larger value of total volume. Since the mass function is nonnegative, $m_0\geq 0$. However, $S^n$ has $m(V) \equiv 0$ because the isoperimetric surfaces are just $n-1$ dimensional spheres. Hence we have
		\[\frac{1}{2}\vol(M) = \int_\gamma \frac{dx}{y} \leq \sup_{\gamma} \int_\gamma \frac{dx}{y} = \frac{1}{2}\vol(S^n)\]
		which completes the proof of Bishop's theorem.
	\end{proof}
	
	\newcommand{\ai}{\alpha}
	\newcommand{\be}{\beta}
	\newcommand{\Ga}{\Gamma}
	\newcommand{\ga}{\gamma}	
	\newcommand{\de}{\delta}
	\newcommand{\De}{\Delta}
	\newcommand{\e}{\epsilon}
	\newcommand{\lam}{\lambda}
	\newcommand{\Lam}{\Lamda}
	\newcommand{\om}{\omega}
	\newcommand{\Om}{\Omega}
	\newcommand{\si}{\sigma}
	\newcommand{\Si}{\Sigma}
	\newcommand{\vp}{\varphi}
	\newcommand{\rh}{\rho}
	\newcommand{\ta}{\theta}
	\newcommand{\Ta}{\Theta}
	\newcommand{\W}{\mathcal{O}}
	\newcommand{\ps}{\psi}
	\newcommand{\mf}[1]{\mathfrak{#1}}
	\newcommand{\ms}[1]{\mathscr{#1}}
	\newcommand{\mb}[1]{\mathbb{#1}}
	\newcommand{\cd}{\cdots}
	\newcommand{\s}{\subset}
	\newcommand{\es}{\varnothing}
	\newcommand{\cp}{^\complement}
	\newcommand{\bu}{\bigcup}
	\newcommand{\ba}{\bigcap}
	\newcommand{\ti}[1]{\tilde{#1}}
	\newcommand{\la}{\langle}
	\newcommand{\ra}{\rangle}
	\newcommand{\ov}[1]{\overline{#1}}
	\newcommand{\no}[1]{\left\lVert#1\right\rVert}
	\newcommand{\du}{^\ast}
	\newcommand{\pf}{_\ast}
	\newcommand{\is}{\cong}
	\newcommand{\n}{\lhd}
	\newcommand{\m}{^{-1}}
	\newcommand{\ts}{\otimes}
	\newcommand{\ip}{\cdot}
	\newcommand{\op}{\oplus}
	\newcommand{\xr}{\xrightarrow}
	\newcommand{\xla}{\xleftarrow}
	\newcommand{\xhl}{\xhookleftarrow}
	\newcommand{\xhr}{\xhookrightarrow}
	\newcommand{\mi}{\mathfrak{m}}
	\newcommand{\wi}{\widehat}
	\newcommand{\sch}{\mathcal{S}}
	\newcommand{\na}{\nabla}
	\newcommand{\N}{\mathbb{N}}
	\newcommand{\R}{\mathbb{R}}
	\newcommand{\Z}{\mathbb{Z}}
	\newcommand{\Q}{\mathbb{Q}}
	\newcommand{\C}{\mathbb{C}}
	\newcommand{\bh}{\mathbb{H}}

	\section{Singular Isoperimetric Hypersurfaces}
	\theoremstyle{plain}
	\newtheorem{lem}[thm]{Lemma}
	\subsection{Regularity and Control on Singular Sets}
	The main complication to using the method of \cite{bray2009penrose} in higher dimensions is that singular isoperimetric hypersurfaces might have singularities. In this section, we estimate the size of small neighborhoods around the singular sets using geometric measure theory. We show that these neighborhoods have small enough area so that carrying out the flow in Section 2 outside these neighborhoods would still give a proof as in the smooth case.
	
	First, we recall the well-known regularity result, such as lemma 3.1 below, regarding isoperimetric hypersurfaces.\begin{lem}
		(Corollary 3.8 in \cite{frankmorgan})
		Let $\Si$ be an $n-1$-dimensional isoperimetric hypersurface in a smooth Riemannian manifold $M$. Then except for a set of Hausdorff dimension at most $n-8,$ $\Si$ is a smooth submanifold of $M$.
	\end{lem}For a detailed discussion of the history of the regularity theorem above and a recent proof, please refer to Morgan's great paper \cite{frankmorgan}. 
	
	Our strategy to deal with the singularities is to control the area of the isoperimetric surface around the singular sets in the following sense.
	\begin{lem}
		For $\Si$ an isoperimetric hypersurface in $M,$ we have the following uniform bound,
		\begin{align*}
		\mathcal{H}^{n-1}(B_\rh(\xi)\cap \Si)\le C\rh^{n-1},
		\end{align*}for some positive constant $C$ depending on only $M$ and $\Si$.
	\end{lem}
	The rest of this section will be dedicated to the proof of Lemma 3.2
	\subsection{Proof of Lemma 3.2}
	The proof is basically a straightforward application of the following monotonicity formula for varifolds in the lecture notes \cite{leonsimon} by Leon Simon.
	\begin{lem}
		(Theorem 17.6 in \cite{leonsimon}) $V$ is an $m$-dimensional varifold in $\R^{m+l}$, with $\mu_V$ associated measure and $H$ generalized mean curvature. Suppose $V$ is contained in an open set $U,$ with Euclidean ball $E_\rh(\xi)\s U$ for some point $\xi\in \Si.$ If $|H|\le \Lambda$, a positive constant, then
		\begin{align*}e^{\Lambda\rh}\rh^{-m}\mu_V(E_\rh(\xi))
		\end{align*}is non-decreasing in $\rh.$
	\end{lem}
	To use Lemma 3.3, we have to first embed $M$ in some $\R^{n+l}$ as an $n$-dimensional submanifold with induced metric and then get a mean curvature bound on the singular soap bubble $\Si$ with respect to the ambient Euclidean space. Embedding is always possible by Nash embedding theorem. For mean curvature bound, we need the following lemma.
	\begin{lem}
		If $\Si\s M$ is a singular soap bubble, i.e., the mean curvature is constant everywhere. If $M$ is isometrically embedded in $\R^{n+l}$, then $\Si$ has bounded mean curvature in $\R^{n+l}$ as well.
	\end{lem}
	\begin{proof}
		Let $\{e_1,\cd,e_n,E_1,\cd,E_l\}$ denote a smooth frame adapted to $M\s \R^{n+l}$ in some neighborhood, with $E_1,\cd,E_l\in T^\perp M.$ By compactness of $M,$ there exists a finite cover $\{U_j\}$ of $M$ so that on each open set $U_j$, such smooth adapted frames exist. Shrinking the neighborhoods $U_j$ if necessary, we have $|\na_{e_j}E_i|\le A$ for some positive constant $A$, all $i,j$ on all neighborhoods $U_k$. Moreover, for every point, let $\{\nu_j\}$ be a frame of $\Si$ adapted to $M,$ with $\nu_n$ the unit normal of $\Si\s M.$ Such pointwise frames exist pointwise except for a codimensional 8 set. Let $H^{M_1\s M_2}$ denote the mean curvature of $M_1$ in $M_2.$ \DeclarePairedDelimiter{\ri}{\la}{\ra}
		We have
		\begin{align*}
		&H^{\Si\s \R^{m+l}}-H^{M\s \R^{m+l}}\\=&-\sum_{j=1}^{n-1}\sum_i\ri{\na_{\nu_j}E_i,\nu_j}E_i-\sum_{j=1}^{n-1}\ri{\na_{\nu_j}\nu_n,\nu_j}\nu_n+\sum_{j=1}^n\sum_i\ri{\na_{\nu_j}E_i,\nu_j}E_i\\
		=&\sum_i\ri{\na_{\nu_n}E_i,\nu_n}E_i+H^{\Si\s M}.
		\end{align*}
		We have
		\begin{align*}
		\left|\sum_i\ri{\na_{\nu_n}E_i,\nu_n}E_i\right|\le& \sum_i \left|\na_{\nu_n} {E_i}\right|\\
		\le&\sum_{j,i} |\ri{\nu_n,e_j}||\na_{e_j}E_i|\\
		\le&\sum_{j,i}|\na_{e_j}E_i|\\
		\le& nlA.
		\end{align*} Since $M$ is compact, $H^{M\s \R^{m+l}}$ is also bounded. Thus, $H^{\Si\s \R^{m+l}}$ is bounded by
		\begin{align*}
		|H|\le \Lambda=\sup\left|H^{M\s \R^{m+l}}\right|+\left|H^{\Si\s M}\right|+nlA,
		\end{align*} except on a codimensional 8 set.
	\end{proof}
	Now, let $\Si$ be a soap bubble, $E_\rh(\xi)$ be the Euclidean $\rh$-ball around $\xi\in \Si$ in $\R^{m+l}$ and $\text{diam}(M)$ be the Euclidean diameter of the embedded $M.$ By Lemma 3.2, we have
	\begin{align*}
	\rh^{-(n-1)}\text{area}(E_\rh(\xi)\cap \Si)\le e^{2\Lambda\cdot\text{diam}(M)}\text{diam}(M)^{-(n-1)}\text{area}(\Si).
	\end{align*}
	Now that $M$ is a Riemannian submanifold of $\R^{m+l},$ the distance on $M$ is larger than the Euclidean distance, so $B_\rh(\xi)\s E_\rh(\xi),$ with $B_\rh(\xi)$ the $\rh$-ball in $M.$ This gives
	\begin{align*}
	&\rh^{-(n-1)}\text{area}(B_\rh(\xi)\cap\Si)\\\le&\rh^{-(n-1)}\text{area}(E_\rh(\xi)\cap \Si)\\\le& e^{2\Lambda\cdot\text{diam}(M)}\text{diam}(M)^{-(n-1)}\text{area}(\Si).
	\end{align*}
	Thus, we can conclude that
	\begin{align*}
	\mathcal{H}^{n-1}(B_\rh(\xi)\cap \Si)\le C\rh^{n-1},
	\end{align*}for some positive constant $C$ depending on only $M,\Si,$ and an embedding of $M$ into Euclidean space.

	\section{Singular case}
	In this section $n\ge 8$, so there may exist a singular set on the isoperimetric surface, hence, we can not define a unit normal vector at the singular set. We choose a cutoff function such that it vanishes at the singular set and equals to 1 outside a small neighborhood of the singular set. Multiplying this cutoff function with the outward unit normal vector, we can construct a geometric flow which fixes the singular set on the isoperimetric surface.
	\begin{thm}
		$A''(V)\le -\frac{1}{A(V)}(\frac{1}{n-1}A'(V)^2+\Ric_0)$, in the sense of comparison function.
	\end{thm}
	\begin{proof}
		Let $\Sigma_{V_0}$ be the isoperimetric surface with respect to the bounding volume $V_0$, assume $\mathcal{S}$ is the singular set of $\Sigma_{V_0}$, then $\mathcal{S}$ is compact. Assume  $\mathcal{R}=\Sigma_{V_0}-\mathcal{S}$.
		
		According to lemma 3.1, $\mathcal{H}^{n-7}(\mathcal{S})=0$, then for any $\delta>0$, there exist
		$S_\delta=\cup_{i}{B_{r_i}(x_i)}$, such that $r_i\le \delta$, $\mathcal{S}\in S_\delta$, $\sum_i r_i^{n-7}\le 1$. Assume $S'_\delta=\cup_{i}{B_{2r_i}(x_i)}$.
		
		Construct a series of smooth functions $\{\eta_i\}$, such that $\eta_i\equiv 1$ on $M-B_{2r_i}(x_i)$; $\eta_i\equiv 0$ on $B_{r_i}(x_i)$; $\eta_i\in [0,1]$, $|\nabla \eta_i|\le C_0/r_i$, $|\Delta \eta_i|\le C_0/r_i^2$.
		
		Let $\tilde{\eta}=\min\{ \eta_i\}$, $\eta=\tilde{\eta}|_{\Sigma_{V_0}}$.
		As $\eta_i$ are Lipschitz functions, then $\eta$ is Lipschitz, so we can define $\nabla_{\Sigma_{V_0}}\eta$ be the gradient of $\eta$ on $\mathcal{R}$.  
		
		We have:
		\begin{equation*}
		|\nabla_{\Sigma_{V_0}} \eta|^2 \le \sum_i  |\nabla_{\Sigma_{V_0}} \eta_i|^2\le \sum_i|\nabla_M \eta_i|^2\le \sum_i C_0^2 r_i^{-2}, a.e.
		\end{equation*}
		Let $U=\mathcal{R}-S'_{\delta}$. $\eta=0$ on $S_\delta\cap \Sigma_{V_0}$, $\eta=1$ on $U$. 
		
		Let $W_\delta=(S'_\delta-S_\delta)\cap\Sigma_{V_0}$, then according to lemma 3.2, $\exists$ $C\ge 0$, $\mathcal{H}^{n-1}(B_{2r_i}(x_i)\cap \Sigma_{V_0})/r_i^{n-1}\le C$, for all $i$. 
		$W_\delta\subset S'_\delta \cap \Sigma_{V_0}$, then: 
		\begin{align*}
		&\mathcal{H}^{n-1}(W_\delta)\le \mathcal{H}^{n-1}(S'_\delta\cap \Sigma_{V_0})
		\\ \le &\sum_i \mathcal{H}^{n-1}(B_{2r_i}(x_i)\cap \Sigma_{V_0})\le \sum_i Cr_i^{n-1}\le C\delta^6.
		\end{align*}
		
		Let the flow be $\vec{\varphi}=\eta \nu$, $\nu$ is outward normal vector on $\mathcal{R}$. We can extend $\eta$ to a neighbourhood of $\Sigma_{V_0}$, such that $\frac{\partial\eta}{\partial t}=0$ on $\Sigma_{V_0}$.  Similar to the smooth case, we still use the notation $\Sigma_{V_0}(t)$ to denote the surface at time $t$ under the flow $\vec{\varphi}$. We use $A_{V_0}(t)$ to denote the area of $\Sigma_{V_0}(t)$ parameterized by $t$, $A_{0}(V)$ is the area of $\Sigma_{V_0}(t)$ parameterized by $V$. 
		
		Recall the second variation formula for a smooth isoperimetric surface: 	\begin{equation*}
		A''_{V_0}(t)=\int_{\Sigma_{V_0}(t)}(-\Delta_{\Sigma_{V_0} (t)}\eta-\eta\|\Pi\|^2-\eta \Ric(\nu,\nu))\eta+H\frac{\partial \eta}{\partial t}+H^2\eta^2dA.
		\end{equation*}
		It is slightly different from equation (1), since the flow here is not unit speed.
		
		In singular case, as $\eta\equiv1$ on $U$, $\eta\equiv0$ on $\Sigma_{V_0}-U-W_\delta$, $\frac{\partial \eta}{\partial t}=0$ on $\Sigma_{V_0}$, we have
		\begin{align*}
		A''_{V_0}(0) =&\int_{U}(-\|\Pi\|^2- \Ric(\nu,\nu)+H^2)dA
		\\&+ \int_{W_\delta} (|\nabla_{\Sigma_{V_0}}{\eta}|^2-\eta^2\|\Pi\|^2-\eta^2 \Ric(\nu,\nu))+H^2\eta^2dA.
		\end{align*}
		For the $W_\delta$ part estimation, $\exists$ constant $C_1$,
		\begin{align*}
		&\int_{W_\delta} (|\nabla_{\Sigma_{V_0}}{\eta}|^2-\eta^2\|\Pi\|^2-\eta^2 \Ric(\nu,\nu))+H^2\eta^2dA
		\\\le &\int_{W_{\delta}}H^2\eta^2+ \sum_i|\nabla_{\Sigma_{V_0}}\eta_i|^2dA
		\\ \le &H^2 \mathcal{H}^{n-1}(W_\delta)
		+\sum_i\int_{\Sigma_{V_0}\cap(B_{2r_i}(x_i)-B_{r_i}(x_i))}|\nabla_{\Sigma_{V_0}}\eta_i|^2dA
		\\ \le& H^2C \delta^6+ \sum_i C_0^2r_i^{-2}\mathcal{H}^{n-1}(B_{2r_i}(x_i)\cap \Sigma)
		\\ \le & H^2C \delta^6 +\sum_i C_0^2Cr_i^{n-3}\le C_1 \delta^4.
		\end{align*}
		Hence, $\displaystyle{A_{V_0}''(0)\le \int_U (-\|\Pi\|^2-\Ric(\nu,\nu))+H^2 dA+O(\delta^4)}$.
		
		Similar to the smooth case, we have formulas for $A_{V_0}'(0
		)$, $V'(0)$, $V''(0)$, $A''_{0}(V_0)$: 
		\begin{equation*}
		A'_{V_0}(0)=\int_{\Sigma_{V_0}} H\eta dA=\int_U HdA+\int_{W_\delta} H\eta dA= HA_0(V_0)+O(\delta^6).
		\end{equation*}
		\begin{equation*}
		V'(0)=\int_{\Sigma_{V_0}}\eta dA=A_0(V_0)+O(\delta^6).
		\end{equation*}
		\[V''(0)=\left.\frac{d}{dt}\right|_{t=0}\int_{\Sigma_{V_0}(t)}\eta dA=\int_{\Sigma_{V_0}}H\eta^2 dA=HA_0(V_0)+O(\delta^6).\]
		\[A'_0(V_0)=\frac{A'_{V_0}(0)}{V'(0)}=H+O(\delta^6).\]
		\begin{equation*}
		A''_{0}(V_0)=\frac{A''_{V_0}(0)-A'_{0}(V_0)(V''(0))}{V'(0)^2}.
		\end{equation*}
		Let $\delta\rightarrow0$, then we have the same formulas for $A''_{0}(V_0)$ as the smooth case.
		
		Therefore, $A''(V)\le -\frac{1}{A(V)}(\frac{1}{n-1}A'(V)^2+\Ric_0)$.
	\end{proof}
	From here, we can follow the proof of Bishop theorem in section 2 for $n\geq8$.
	\section{Afterward}
	In previous sections, we have proved Bishop's theorem using singular isoperimetric hypersurface. It might seem at first an overkill to prove Bishop's theorem using advanced machinery like geometric measure theory as in our proof, while a simple one using geodesic balls is already well-known. However, our proof of Bishop's theorem serves as a starting point of a grand scheme of isoperimetric surface techniques. We will illustrate the power of isoperimetric surface techniques by presenting the following scalar curvature comparison theorem. In fact the first author first proved the following theorem in \cite{bray2009penrose}, and then discovered the proof of Bishop's theorem in this paper as a byproduct. 	It's remarkable that, as of today, more than twenty years after the first author proved the above theorem, the only known proofs all use isoperimetric surface techniques.
	\begin{thm}
		(Football Theorem)		Let $(S^3,g_0)$ be the constant curvature metric on $S^3$ with scalar curvature $R_0,$ Ricci curvature $\Ric_0\ip g_0$, and volume $V_0.$ There exists a positive constant $\varepsilon_0<1$ so that for any complete smooth Riemannian manifold $(M^3,g)$ of volume $V$ satisfying 
		\begin{align}
		R(g)\ge &R_0,\\
		\Ric(g)\ge \varepsilon_0\ip & \Ric_0\ip g,
		\end{align}we have
		\begin{align*}
		V\le V_0.
		\end{align*}
		Alternatively, if
		\begin{align*}
		R(g)\ge &R_0,\\
		\Ric(g)\ge \varepsilon\ip & \Ric_0\ip g,
		\end{align*}with $\varepsilon>0,$ then 
		\begin{align*}
		V\le \ai(\varepsilon)V_0,
		\end{align*}
		where
		\begin{align*}
		\ai(\varepsilon)=\sup_{\frac{4\pi}{3-2\varepsilon}\le z\le 4\pi}\frac{1}{\pi^2}\begin{pmatrix}
		\int_0^{y(z)}\left(36\pi-27(1-\varepsilon)y(z)^{\frac{2}{3}}-9\varepsilon\ip x^{\frac{2}{3}}\right)^{-\frac{1}{2}}dx\\
		+\int_{y(z)}^{z^{\frac{3}{2}}}\left(36\pi-18(1-\varepsilon)y(z)^{-\frac{1}{3}}-9x^{\frac{2}{3}}\right)^{-\frac{1}{2}}dx
		\end{pmatrix},
		\end{align*}with $$y(z)=\frac{z^{\frac{1}{2}(4\pi-\varepsilon)}}{2(1-\varepsilon)}.$$ Furthermore, the expression of $\ai(\varepsilon)$ is sharp.
	\end{thm}
	\begin{figure}[H]
	    \centering
	    \includegraphics[width=0.6\linewidth]{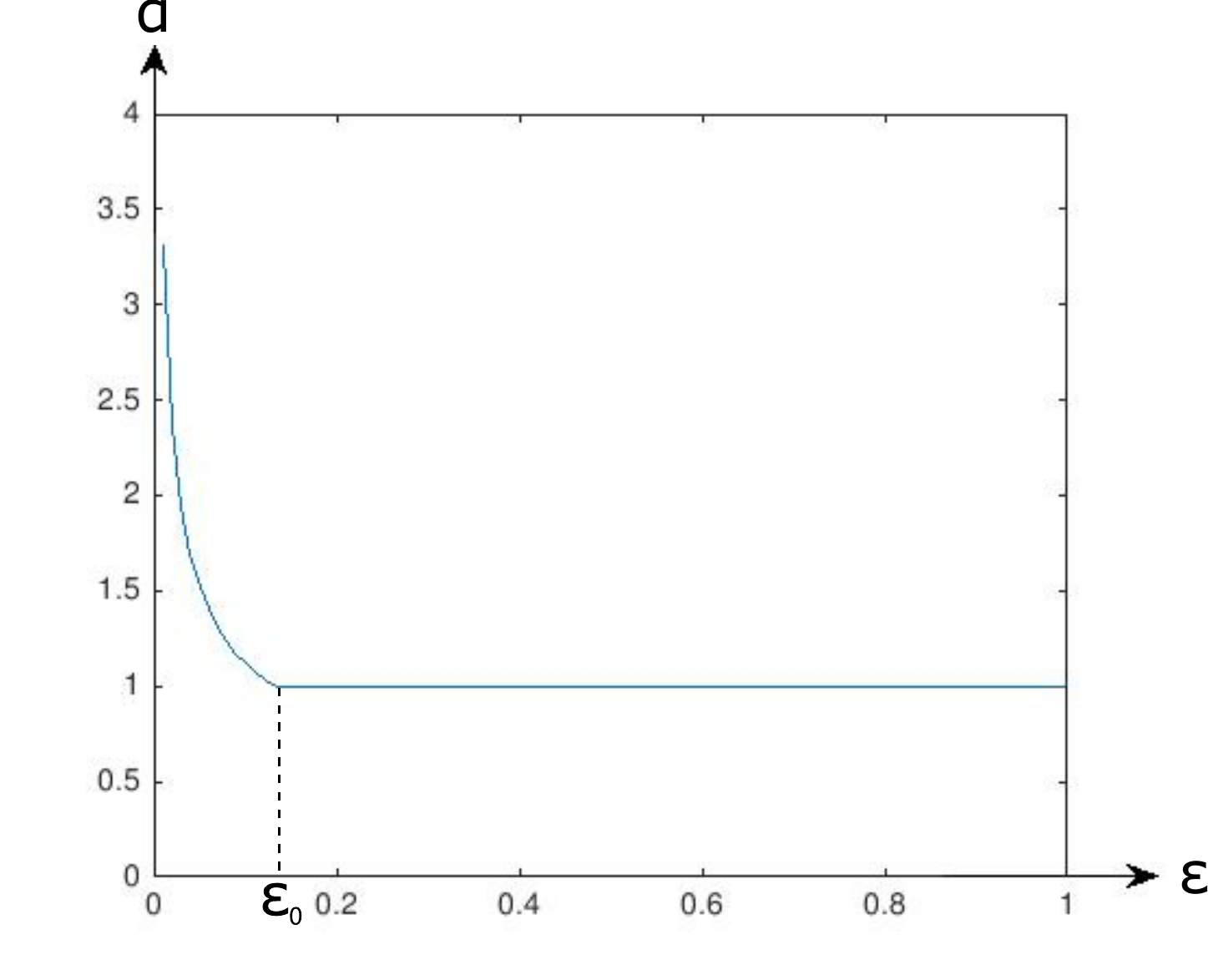}
	    \caption{The graph of $\alpha(\varepsilon)$. When $0<\varepsilon<\varepsilon_0$, then $\alpha(\varepsilon)>1$, it can be achieved by the manifold shown in Figure 5. When $\varepsilon_0\le \varepsilon\le 1$, then $\alpha(\varepsilon)=1$, it can be achieved by the sphere with constant curvature metric.}
	    \label{fig4}
	\end{figure}
	The first part of the theorem can be seen as a normalized version of the second part. The theorem is sharp in the sense both bounds (5) and (6) can almost be achieved. As in the following picture, there exist football-like manifolds with pointy ends, American football-like, to be precise, that achieve equality in the bounds (5) and (6), hence the name Football theorem.
	\begin{figure}[H]
		\centering
		\includegraphics[width=0.6\linewidth]{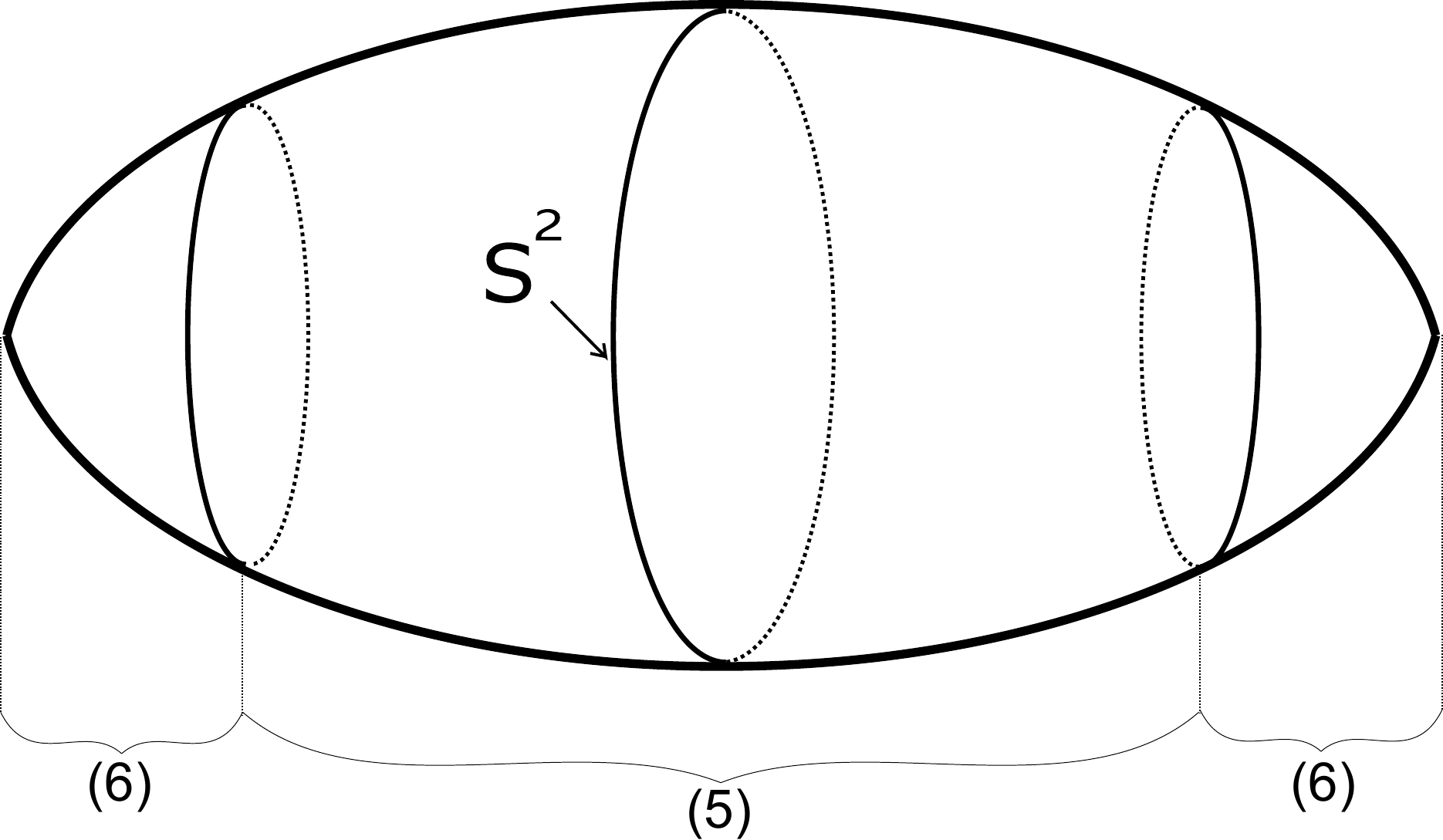}
		\caption{The football in Football Theorem. The portion with label (5) lying below reaches equality in inequality (5), while the portion with label (6) lying below reaches equality in inequality (6).}
		\label{fig}
	\end{figure}
	Moreover, the Ricci curvature lower bound (6) cannot be dispensed with, since there exist counterexamples with positive scalar curvature and arbitrarily large volume. The long cylinder $[0,N]\times S^{n-1},$ with $N>0,$  is such a counterexample. 
	\begin{figure}[H]
		\centering
		\includegraphics[width=0.6\linewidth]{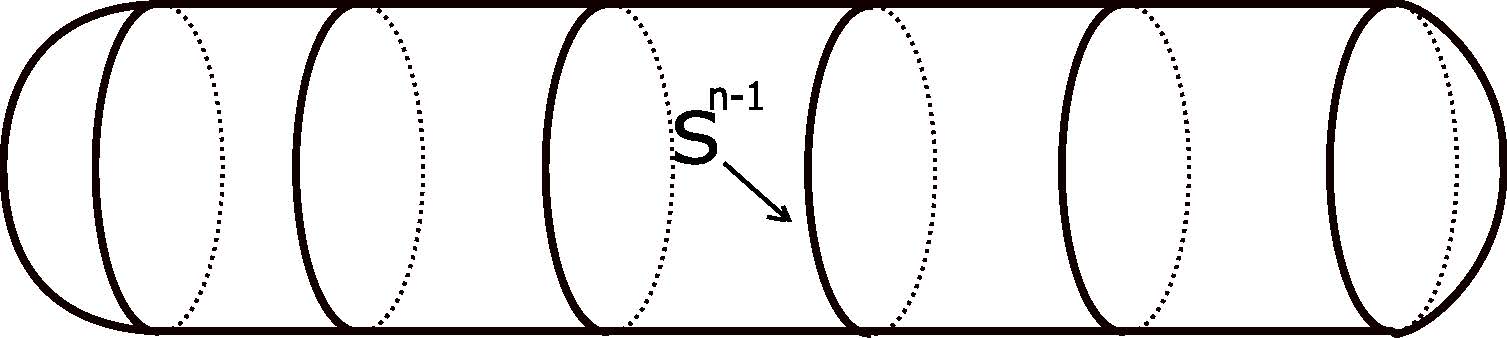}
		\caption{The counterexample when we don't bound Ricci curvature suitably. The cylinder satisfies $\Ric(g)\ge 0.$}
		\label{fig:my_label}
	\end{figure}As illustrated in Figure 6, where every vertical ellipse represents a $S^{n-1}$, the cylinder has the same scalar curvature as the sphere but can have zero Ricci curvature for some pairs of vectors. We can see that the volume of the cylinder has no upper bound since we can make $N$ as large as we want.
	
	Regarding the constant $\varepsilon_0,$ numerical evidence suggests $0.134<\varepsilon_0<0.135.$ Matthew Gursky and Jeff Viaclovsky proved the bound $\varepsilon_0\le\frac{1}{2}$ in \cite{jeffv}.
	
	Also, the dimension $3$ is essential for the proof of theorem, since Gauss-Bonnet and Gauss-Codazzi are both applied to isoperimetric surfaces to utilize the bounds on scalar curvature. However, we believe that the theorem could be extended to high dimensions as in the following conjecture,
	
	\begin{conj}(\cite{bray2009penrose}).
		Let $(S^n,g_0)$ be the constant curvature metric on $S^n$ with scalar curvature $R_0,$ Ricci curvature $\Ric_0\ip g_0$, and volume $V_0.$ There exists a positive constant $\varepsilon_0<1$ so that for any complete smooth Riemannian manifold $(M^n,g)$ of volume $V$ satisfying 
		\begin{align}
		R(g)\ge &R_0,\\
		\Ric(g)\ge \varepsilon_0\ip & \Ric_0\ip g,
		\end{align}we have
		\begin{align*}
		V\le V_0.
		\end{align*}
	\end{conj}
	Now, we will sketch original the proof of Theorem 5.1 in \cite{bray2009penrose} in the following. Alternative exposition can also be found in the survey \cite{brendle} by Simon Brendle.
	\subsection{A Sketch of Proof of Theorem 5.1}
	As we have mentioned, the main ingredient will be isoperimetric surface technique. Using exactly the same reasoning as in Section 2, we can turn the Ricci curvature lower bound (6) into the following ordinary differential inequality
	\begin{align}\label{odi1}
	\boxed{A''{(V)}\le -\frac{1}{A(V)}\left(\frac{1}{2}A'(V)^2+\varepsilon_0\ip \Ric_0\right),}
	\end{align}where all the definitions are the same as in Section 2.
	Now, we are left with the scalar curvature lower bound to deal with. As before
	\begin{align*}
	A_{0}(V_0)^2A''_{0}(V_0)=\int_{\Si_{V_0}}^{}-\no{\Pi}^2-\Ric(\nu,\nu).
	\end{align*}
	By Gauss-Codazzi equations, we get
	\begin{align*}
	\Ric(\nu,\nu)=\frac{1}{2}R(g)-K+\frac{1}{2}H^2-\frac{1}{2}\no{\Pi}^2,
	\end{align*}where $K$ and $H$ are the Gauss curvature and mean curvature of $\Si(V_0)$, respectively. Substituting, we get
	\begin{align*}
	A_{0}(V_0)^2A''_{0}(V_0)=\int_{\Si_{V_0}}-\frac{1}{2}R(g)+K-\frac{1}{2}H^2-\frac{1}{2}\no{\Pi}^2.
	\end{align*}
	To get rid of $K,$ we need some information on the topology of $\Si_{V_0}$. Indeed, $\Si_{V_0}$ has only one connected component, since otherwise we can consider a flow on $\Si_{V_0}$ which is flowing in on one component while flowing out in another. Then all of the surfaces of the family contain the same volume, while by inequality (\ref{odi1}), the second derivative of the area is negative. Thus, $\Si_{V_0}$ doesn't minimize area, which is a contradiction. Hence, by Gauss-Bonnet, we have
	\begin{align*}
	\int_{\Si_{V_0}}K=2\pi\chi(\Si_{V_0})\le 4\pi.
	\end{align*}
	Since $R\ge R_0$ and $\no{\Pi}^2\ge \frac{1}{2}H^2,$ we have
	\begin{align*}
	A_{0}(V_0)^2A''_{0}(V_0)\le& 4\pi-\int_{\Si_{V_0}}(\frac{1}{2}R_0+\frac{3}{4}H^2)\\
	=&4\pi-A_{0}(V_0)\left(
	\frac{1}{2}R_0+\frac{3}{4}H^2\right).
	\end{align*}
	We deduce that
	\begin{align*}
	A''_{0}(V_0)\le \frac{4\pi}{A_{0}(V_0)^2}-\frac{1}{A_{0}(V_0)}\left(\frac{3}{4}A'_{0}(V_0)^2+\frac{1}{2}R_0 \right).
	\end{align*}
	As before, $A(V_0)=A_{0}(V_0)$, and $A(V)\le A_{0}(V)$, so we have
	\begin{align}
	\boxed{A''(V)\le \frac{4\pi}{A(V)^2}-\frac{1}{A(V)}\left(\frac{3}{4}A'(V)^2+\frac{1}{2}R_0\right),}
	\end{align}
	in the sense of comparison functions.
	
	Now that we have the two ordinary differential inequalities (9) and (10), the rest is in the same spirit as Section 2, that is, turning these inequalities into the bounds we want. We will omit the details here. Interested readers can consult the first author's thesis \cite{bray2009penrose}.

\end{document}